\documentclass[12pt]{article}
\usepackage{amssymb,amscd,amsmath,amsthm} 
\usepackage{latexsym}
\usepackage{graphicx}
\usepackage{amssymb}

\DeclareFontFamily{OT1}{pzc}{}
\DeclareFontShape{OT1}{pzc}{m}{it}{<-> s * [1.10] pzcmi7t}{}
\DeclareMathAlphabet{\mathpzc}{OT1}{pzc}{m}{it}

\theoremstyle{plain}

\newtheorem{thm}{Theorem}[section]
\newtheorem{lem}[thm]{Lemma}

\newtheorem{prop}[thm]{Proposition}
\newtheorem{conj}[thm]{Conjecture}

\begin{document}

\pagenumbering{arabic} \setcounter{page}{1}

\title{Sum of dilates in vector spaces}

\author{Antal Balog\thanks{The first author's research was supported by the Hungarian National Science Foundation Grants K81658 and K104183.} \\ 
Alfr\'ed R\'enyi Institute of Mathematics \\
Budapest, P.O.Box 127, 1364--Hungary \\
balog@renyi.mta.hu \\
\and
George Shakan\thanks{Research conducted while the author enjoyed the hospitality of the Alfr\'{e}d R\'{e}nyi 
Institute of Mathematics, and benefited from the OTKA grant K109789.}\\
UIUC Department of Mathematics \\
Urbana, IL 61801, USA \\
shakan2@illinois.edu}

\maketitle


\begin{abstract}
Let $d\geq 2$, $A \subset \mathbb{Z}^d$ be finite and not contained in a translate of any hyperplane, and $q \in \mathbb{Z}$ such that $|q| > 1$. We show $$|A+ q \cdot A| \geq (|q|+d+1)|A| - O_{q,d}(1).$$ 
\end{abstract}

\section{Introduction}

Let $A$ and $B$ be finite sets of real numbers. The sumset and the product set of $A$ and $B$ are defined by $$A + B = \{a + b : a \in A, \ b \in B \},$$
$$A \cdot B = \{ab : a \in A, \ b \in B\}.$$ For a real number $d \neq 0$ the dilation of $A$ by $d$ is defined by 
$$d \cdot A = \{d\} \cdot A =  \{da : a \in A\},$$ while for any real number $x$, the translation of $A$ by $x$ is defined by 
$$x + A = \{x\} + A = \{x + a : a \in A\}.$$The following (actually more)  was shown in \cite{Ba}.

\begin{thm}\label{Ba}\cite{Ba} Let $q \in \mathbb{Z}$. Then there is a constant $C_q$ such that every finite $A \subset \mathbb{Z}$ satisfies
\begin{equation}\label{1}|A+ q \cdot A| \geq (|q|+1) |A| - C_q,\end{equation}
\end{thm} This was obtained after the works of \cite{Bu, Ci, Du, Ha, Lj}. The reader is invited to see the introductions of \cite{Ba} and \cite{Bu} for a more detailed introduction to this problem. 

 For a finite $A \subset \mathbb{Z}^d$, we say the {\em rank} of $A$ is the smallest dimension of an affine space that contains $A$. When $A$ is a set of high rank, one might expect to be able to improve the lower bound in \eqref{1}, which is the goal of our current note. Ruzsa proved the following in \cite{Ru}. \begin{thm}\label{Ru}\cite{Ru} Let $A,B \subset \mathbb{Z}^d$ be finite such that $A+B$ has rank $d$ and $|A| \geq |B|$. Then $$|A+B| \geq |A| +d|B| -  \frac{d(d+1)}{2}.$$ \end{thm}

Let $A \subset \mathbb{Z}^d$ be finite of rank $d$ and $q$ be an  integer. The main objective here is to improve upon \eqref{1} and Theorem \ref{Ru} in the case $B = q \cdot A$. In this note $O(1)$ will always depend on the relevant $d$ and $q$. Our main theorem is the following.

\begin{thm}\label{main}
Let $A \subset \mathbb{Z}^d$ of rank $d\geq 2$ and $|q| > 1$ be an integer. Then $$|A+q \cdot A| \geq (|q|+d+1)|A| - O(1).$$
\end{thm}

The authors would like to thank Imre Ruzsa for drawing our attention to the current problem. We remark that we do not believe even the multiplicative constant of $(|q|+d+1)$ is the best possible, and we now present our best construction.  For $1 \leq i \leq d$, let $e_i$ be the standard basis vectors of $\mathbb Z^d$. For $N \in \mathbb{Z}$, consider $$A_N = \{e_1 , \ldots , e_d\} \cup \{n e_1 : 0 < n < N  \ , \ n \in \mathbb{Z} \}.$$ It is easy to see that \begin{equation}\label{example} |A_N + q \cdot A_N| \leq (q+2d-1)|A_N| - (d-1)( |q| - 2(d-1) + 1) \end{equation} This shows that Theorem \ref{main} is the best possible up to the additive constant for $d=2$. We are also able to handle the case $d = 3$.

\begin{thm}\label{d=3}
Let $A \subset \mathbb{Z}^3$ be finite of rank $3$ and $|q| > 1$. Then $$|A+q \cdot A| \geq (|q|+5)|A| - O(1).$$
\end{thm}

Furthermore, we can prove the following bound for all $q$, and this is  best possible, up to the additive constant, when $|q| = 2$. One can check the example for \eqref{example} to see that $$|A_N \pm 2 \cdot A_N| = (2d+1)|A_N| - d(d+1).$$

\begin{thm}\label{q=2}
Let $A \subset \mathbb{Z}^d$ be finite of rank $d$ and $|q| > 1$. Then $$|A+q \cdot A| \geq (2d+1)|A| - d(d+1)^2/2.$$
\end{thm}

Our basic intuition is that to minimize $|A+q \cdot A|$ one should choose $A$ to be as close to a one dimensional set as possible. One should proceed with caution with this intuition because when $q=-1$, a clever construction in \cite{St} shows that this is not the best strategy. Nevertheless, given the evidence of Theorem \ref{d=3} and Theorem \ref{q=2}  we present the following conjecture. 

\begin{conj}\label{mainconj} Suppose $A \subset{Z}^d$ is finite of rank $d$ and $q$ is an integer with absolute value bigger than 1. Then $$|A+q \cdot A| \geq (|q|+2d-1)|A| - O(1).$$
\end{conj}

We remark that the cases $A+A$ and $A-A$ have different behavior. Theorem \ref{Ru}, which in the case $B =A$ was proved by Freiman in \cite{Fr}, says that $|A+A| \geq (d+1)|A| - d(d+1)/2$. This is the best possible due to \eqref{example}, which shows Theorem \ref{main} is false with $q=1$. The reason that one can improve when $q \neq 1$ is simply that in $A+A$, the roles of the summands are interchangeable, while in the case $A+q \cdot A$, the roles of $A$ and $q \cdot A$ are not interchangeable. We have already mentioned that there is a tricky construction in \cite{St}, which shows $|A-A|$ can be as small as $(2d-2 + \frac{1}{d-1})|A| - (2d^2 - 4d + 3)$. In the same paper, the author conjectures that this is the best possible. It is curious that best known lower bound is $|A-A| \geq (d+1)|A| - d(d+1)/2$.  The case $q = -1$ is also different in the sense that it is important that when $|q| > 1$, we can split $A$ into cosets modulo $q \cdot \mathbb{Z}^d$. This will be seen in our argument below. 

Let $L : \mathbb{Z}^d \to \mathbb{Z}^d$ be a linear transformation. In this note we are primarily concerned with $|A+ L  A|$ where $L$ is a scalar multiple of the identity. The study of other choices of $L$ would be natural, but we do not do it here. 

\section{Proof of Theorems \ref{main} and \ref{q=2}}

Fix $A \subset \mathbb{Z}^d$ of rank $d\geq 2$ and an integer $q$ that is bigger than 1 in absolute value. Since the rank of $A$ is $d$, we must have that $A$ contains at least $(d+1)$ elements. We first partition $A$ into its intersections with cosets of the lattice $q \cdot \mathbb{Z}^d$. Note there are $|q|^d$ such cosets.

Let $$A  = \bigcup_{i=1}^r A_i , \ \ A_i = a_i+ q \cdot A'_i ,\ \  a_i \in \{0,\dots,|q|-1\}^d , \ \  A'_i \neq \emptyset,$$ where the unions are disjoint. We obtain the preliminary estimate

\begin{lem}\label{FD} Let $A \subset \mathbb{Z}^d$ and $q \in \mathbb{Z}$ such that $|q| > 1$. Suppose that $A$ intersects $r$ cosets of the lattice $q \cdot \mathbb{Z}^d$. Then $$|A+ q \cdot A| \geq (d+r)|A| - rd(d+1)/2.$$
\end{lem}
\begin{proof}
 Using Theorem \ref{Ru}, we obtain
\begin{align*}
|A + q \cdot A| & = \displaystyle \sum_{i=1}^r |A_i + q \cdot A| \\
& \geq \displaystyle \sum_{i=1}^r \left( d|A_i| + |A| - \frac{d(d+1)}{2} \right)\\
& = (d+r)|A| - rd(d+1)/2.
\end{align*}
\end{proof}

We say that $A$ is fully distributed (FD) modulo $q \cdot \mathbb{Z}^d$ if $A$ intersects every coset of $q \cdot \mathbb{Z}^d$. Note that if $A$ is FD modulo $q \cdot \mathbb{Z}^d$ then Theorem \ref{main} and Conjecture \ref{mainconj} are far from optimal.

We now describe the process of reducing $A$. Applying an invertible linear transformation to $A$ does not change $|A+q\cdot A|$. Suppose there is some $a \in A$ such that the lattice $\langle A - a \rangle_{\mathbb{Z}} = \Gamma$ is a non--trivial sublattice of  $\mathbb{Z}^d$. Let $L: \mathbb{Z}^d \to \mathbb{Z}^d$ be a linear transformation that maps the standard basis vectors to the basis vectors of $\Gamma$, that is $\Gamma=L\mathbb{Z}^d$. Since $A$ has rank $d$, $L$ is invertible. Then we may replace $A$ with $L^{-1} (A - a)$. Note that $L^{-1} (A - a) \subset \mathbb{Z}^d$ since $A  \subset a+L \mathbb{Z}^d$. Since $1 < {\rm det}(L) \in \mathbb{Z}$, each reduction reduces the volume of the convex hull of $A$ by at least $\frac{1}{2}$. The volume of the convex hull of $A$ is always bounded from below by the volume of the $d$-dimensional simplex so eventually this process must stop. Thus we may assume $\langle A - a \rangle_{\mathbb{Z}} = \mathbb{Z}^d$ for all $a \in A$. Then it follows that we have for all $ 1 \leq i \leq r$,  \begin{equation}\label{reduced}\mathbb{Z}^d = \langle A - a \rangle_{\mathbb{Z}} \subset  \langle a_1 - a_i , \ldots , a_r - a_i , q e_1 , \ldots , q e_d \rangle_{\mathbb{Z}}\subset\mathbb{Z}^d.\end{equation}  Here we used that if $x \in A - a$ and $a\in A_i$, then for some $1\leq j\leq r$ we have $x \in a_j-a+q \cdot A'_j \subset \langle a_j-a_i, q e_1 , \ldots , qe_d \rangle_{\mathbb{Z}}.$ We say $A$ is {\it reduced} if $A$ satisfies \eqref{reduced}. 

\begin{proof}[Proof of Theorem \ref{q=2}]
By the discussion above, we may assume $A$ is reduced. We first aim to show that a reduced set must intersect at least $d+1$ cosets of $q \cdot \mathbb{Z}^d$, and then we will appeal to the argument of Lemma \ref{FD}. 


Observe that the linear combinations of $a_1-a_1,\dots,a_r-a_1$ can only take at most $|q|^{r-1}$ different vectors mod $q \cdot \mathbb{Z}^d$. Since $A$ is reduced, by \eqref{reduced}, these vectors must intersect every coset modulo $q \cdot \mathbb{Z}^d$. Thus we have that $|q|^{r-1} \geq |q|^d$, and so $r-1 \geq d$. 

 Then by Theorem \ref{Ru}, we find \begin{align*} |A+q \cdot A| &\geq \left( \displaystyle \sum_{i=1}^d |A_i + q \cdot A| \right) + |(A \setminus (\bigcup_{i=1}^d A_i) + q \cdot A| \\
&\geq \left( \sum_{i=1}^d (d|A_i| + |A| - d(d+1)/2) \right)+ d|A \setminus (\bigcup_{i=1}^d A_i) | + |A| - d(d+1)/2 \\
& = (2d+1)|A| - d(d+1)^2/2. \end{align*} \end{proof}

We now focus our attention to the proof of Theorem \ref{main}. We start with a special case. Recall that we assume $d \geq 2$.

\begin{lem}\label{special} Suppose $A$ is contained in $d$ parallel lines. Then $|A+ q \cdot A| \geq (|q| +2 d - 1) |A| - O(1)$. 
\end{lem}
\begin{proof}
Suppose $A$ is contained in $x_1 + \ell , \ldots , x_d + \ell$ for some 1 dimensional subspace $\ell$. After a translation of $A$  by $-a$ for an element $a\in A$ we can suppose $x_1=0$ and without loss of generality, we may suppose  $x_2,\dots, x_d$ are elements of $\ell^\bot\cong\mathbb R^{d-1}$. Moreover, we have that $x_2,\dots, x_d$ are linearly independent over $\mathbb R$ since $A$ has rank $d$. This implies that for all $1 \leq i , j \leq d$, the lines $(x_i + \ell) + q \cdot (x_j + \ell)$ are pairwise disjoint. For $1 \leq i \leq d$, let $B_i := A \cap (x_i + \ell)$. It follows, using \eqref{1} that 
\begin{align*}|A+q \cdot A| &= \sum_{i = 1}^d \sum_{j=1}^d |B_i+ q \cdot B_j| \\
& = \sum_{i=1}^d \left(|B_i + q \cdot B_i| + \sum_{j \neq i} |B_i + q \cdot B_j|\right) \\
& \geq \sum_{i=1}^d \left( ((|q|+1) |B_i| - O(1))+ \sum_{j \neq i}  (|B_i| + |B_j| - 1)\right)\\
& = (|q| +2d-1)|A| - O(1).
\end{align*}
\end{proof}

We remark that the lack of a satisfactory higher dimensional analog of Lemma \ref{special} is essentially what blocks us from improving the multiplicative constant in Theorem \ref{main}. 
We prove Theorem \ref{main} by induction on $d$ starting from  $d=2$ (the statement is not true for $d=1$). Note that the proof of the next  lemma does not use the induction hypothesis for $d=2$, only for $d\geq 3$.

\begin{lem}\label{f}
Let $B\subset A$ and suppose that the rank of $B$ is $1\leq f < d$. Then $$|B + q \cdot A| \geq (|q|+d+1)|B| - O(1),$$ or $A$ is contained in $d$ parallel lines. 

\end{lem}
\begin{proof}
Note that the rank of $B + q \cdot B$ is also $f$. Since $B + q \cdot A$ is of rank $d$, we may find an $x \in A$ such that $B + q x$ is not in the affine span of $B + q \cdot B$. Thus $B +q \cdot B$ and $B + qx$ are disjoint. The rank of $B \cup \{x\} + q \cdot (B \cup \{x\})$ is $f+1$. We may repeat this process with $B \cup \{x\} + q \cdot (B \cup \{x\} )$ in the place of $B + q \cdot B$, and so on, a total of $(d-f)$ times. Thus we find $x_1 , \ldots , x_{(d-f)} \in A $ such that $B + q \cdot B , B + q x_1 , \ldots ,B + q x_{(d-f)}$ are pairwise disjoint. When $f\geq 2$ (so $d\geq 3$) we use the induction hypothesis, that is Theorem \ref{main} for the sum $B+q\cdot B$ where $B$ is of rank $2\leq f<d$ to get 
$$|B + q \cdot A| \geq |B + q \cdot B| + \sum_{j=1}^{d-f} |B +q x_j| \geq (|q| + d + 1)|B| - O(1).$$

Now we handle the case $f = 1$ (this is the only possibility when $d=2$), in this case we do not use the induction hypothesis. $B$ is contained in a line. We may suppose $A$ is not contained in $d$ parallel lines. We proceed as above to find $x_1 , \ldots , x_{d-1}$ such that $B + q \cdot B , B + q x_1 , \ldots ,B + q x_{(d-1)}$ are pairwise disjoint. Since $A$ is not contained in $d$ parallel lines, we may find an $x_d \in A$ such that $B + q x_d$ is disjoint from all $B + q \cdot B , B + q x_1 , \ldots ,B + q x_{(d-1)}$. It follows from Theorem \ref{Ba} applied to the sum $B+q\cdot B$ that $$|B + q \cdot A| \geq |B + q \cdot B| + \sum_{j=1}^{d} |B +q x_j| \geq (|q| +d+ 1)|B| - O(1).$$
\end{proof}

The next lemma is a higher dimensional analog of Lemma 3.1 in \cite{Ba}. 
\begin{lem}\label{dist}
Let $1 \leq i \leq r$. Then either $A'_i$ is FD modulo $q\cdot \mathbb Z^d$ or $$|A_i + q \cdot A| \geq |A_i + q \cdot A_i| + \min_{1 \leq w \leq r} |A_w|.$$
\end{lem}
\begin{proof}
Suppose $$|A_i + q \cdot A| < |A_i + q \cdot A_i| + \min_{1 \leq w \leq r} |A_w|.$$ Fix $1 \leq w \leq r$. Since $A_w \subset A$, we find that $$|(A_i + q \cdot A_w) \setminus (A_i + q \cdot A_i)| < |A_w|.$$ Translation by $-a_i$ and dilation by $\frac{1}{q}$ reveals that $$|(a_w - a_i + A'_i + q \cdot A'_w) \setminus (A'_i + q \cdot A'_i)| < |A'_w|.$$ Thus for any $x \in A'_i$ there is a $y \in A'_w$ such that $a_w - a_i + x + qy \in A'_i + q \cdot A'_i$. It follows that there is a $x' \in A'_i$ such that $a_w - a_i +x \equiv x' \mod q \cdot \mathbb{Z}^d.$ We may repeat this argument with $x'$ in the place of $x$, and so on, and for each $1 \leq w \leq r$ to obtain that for any $u_1 , \ldots , u_r \in \mathbb{Z}$ there is a $x'' \in A'_i$ such that 
$$u_1(a_1-a_i)+\dots +u_r(a_r-a_i)+x \equiv x'' \mod q \cdot\mathbb{Z}^d.$$ 
Since $A$ is reduced, this describes all of the cosets modulo $q \cdot \mathbb{Z}^d$ and it follows that $A'_i$ is FD mod $q \cdot \mathbb{Z}^d$. 
\end{proof}

We are now ready to prove Theorem \ref{main}. We start with $|A+q \cdot A| \geq |A|$ and improve upon the multiplicative constant iteratively. 
%

\begin{prop}\label{proof} Suppose $A \subset \mathbb{Z}^d$ such that $A$ has rank $d$. Let $q \in \mathbb{Z}$ such that $|q| > 1$. Then for every $|q|+d+1 \leq m \leq (|q| + d + 1)^2$, one has $$|A+q \cdot A| \geq \frac{m}{|q|+d+1} |A| - O(1),$$ where $O(1)$ also depends on $m$. 
\end{prop}

\begin{proof} Observe that $m = (|q|+d+1)^2$ is precisely Theorem \ref{main}. For convenience, set $S = |q|+d+1$. We prove by induction on $m$, where $|A+q\cdot A| \geq |A|$ trivially starts the induction. Suppose now that Proposition \ref{proof} is true for a fixed $S \leq m < S^2$, and we prove it for $m+1$.

If $A$ is contained in $d$ parallel lines, then Lemma \ref{special} immediately implies Theorem \ref{main}, and so Proposition \ref{proof} is especially true for $m+1$ as well. Thus we may assume $A$ is not contained in $d$ parallel lines. 

Consider a set $B\subset A$. If it is $1\leq f < d$ dimensional, then Lemma \ref{f} shows that $|B + q \cdot A| \geq S |B| -O(1)$. If $B$ is $d$ dimensional, then by the induction hypothesis on $m$, we have 
$|B + q \cdot A| \geq |B+q\cdot B| \geq \frac{m}{S} |B|- O(1)$. In either case, using that $m < S^2$, we have 
\begin{equation}\label{comp} |B + q \cdot A| \geq \frac{m}{S} |B| - O(1).\end{equation}

First, assume there is an $1 \leq i \leq r$ such that $|A_i| \leq \frac{1}{S} |A|$. We have by \eqref{comp} and Theorem \ref{Ru}, that

\begin{align*}
| A + q \cdot A| &\geq |A_i + q \cdot A| + |(A \setminus A_i) + q \cdot A | \geq \\
\geq & |A_i| + |A| - 1  + \frac{m}{S} (|A| - |A_i|) - O(1) \geq \frac{m+1}{S} |A| - O(1).
\end{align*} 

Thus we may assume that every $A_i$ has more than $\frac{1}{S}|A|$ elements. 

Suppose now that every $A_i$ is strictly less than $d$ dimensional. Then Lemma \ref{f} shows that 
\begin{align*}
|A + q \cdot A| &= \sum_{i=1}^r |A_i + q \cdot A| \geq \sum_{i=1}^r ((|q|+d+1)|A_i| - O(1)) \\ &= (|q|+d+1)|A| - O(1) \geq \frac{m+1}{S} |A| -O(1).
\end{align*}

Thus we may assume that there is an $A_i$ that is $d$ dimensional.
If the corresponding $A_i^{\prime}$ is not FD modulo  $q \cdot \mathbb{Z}^d$, then by Lemma \ref{dist}, \eqref{comp}, and by the induction hypothesis for $A_i$ we have

\begin{align*}
|A + q \cdot A| &\geq | A_i + q \cdot A| + | (A \setminus A_i) + q \cdot A| \\
&\geq  |A_i + q \cdot A_i| + \min_{1 \leq w \leq r} |A_w| + \frac{m}{S} (|A| - |A_i|) - O(1) \\
&\geq \frac{m}{S} |A_i|  - O(1) + \frac{1}{S} |A| + \frac{m}{S} (|A| - |A_i|) - O(1) = \frac{m+1}{S} |A| - O(1).
\end{align*} 
Similarly if $A_i^{\prime}$ is FD mod $q \cdot \mathbb{Z}^d$ (and $A_i$ is $d$ dimensional) then by Lemma \ref{FD} and \eqref{comp}  we have

\begin{align*}
|A + q \cdot A| &=  |A_i + q \cdot A| + |A\setminus A_i + q\cdot A| \geq |A'_i + q \cdot A'_i| + |A\setminus A_i + q\cdot A|\\
& \geq (|q|^d + d) |A'_{i}| - O(1) + \frac{m}{S} (|A|-|A_i|) - O(1) \geq \frac{m+1}{S} |A| - O(1).
\end{align*} 
Note that the only place where we have used the hypothesis of the induction on $d$ is the $f\geq 2$ case of the proof of Lemma \ref{f}, what we do not use when $d=2$ thus this argument also proves Theorem \ref{main} in that case.
\end{proof}

\section{Proof of Theorem \ref{d=3}}
Let $A \subset \mathbb{Z}^3$ of rank 3 and $q$ be a positive integer such that $|q| > 1$.

The proof of Theorem \ref{d=3} is almost identical to that of Theorem \ref{main}. The only difference is that we have to strengthen Lemma \ref{special}. The reader is invited to check that it is enough to prove Lemma \ref{special} in the case where $d=3$ and $A$ is contained in two parallel planes or 4 parallel lines and then the proof of Theorem \ref{main} goes through in an identical manner. Indeed, if one was able to prove Theorem \ref{main} in the special cases for each $1 \leq f \leq d-1$, and $A$ is contained in $2(d-f)$ translates of a $f$-dimensional subspace, then this along with the proof of Theorem \ref{main} would imply Conjecture \ref{mainconj}.

\begin{lem}\label{previous}
Suppose $A$ is contained in two parallel hyperplanes. Then $$|A+q \cdot A| \geq (|q|+5)|A| - O(1).$$
\end{lem}
\begin{proof}
Suppose $A \subset H \cup (H +x)$ for some hyperplane $H$ and some $x \in \mathbb{Z}^3$. Since $|q| > 1$, we have that $$(H +q \cdot H), (H+x + q \cdot H) ,( H+ q \cdot (H+x)), ((H+x) + q \cdot (H+x)),$$ are disjoint Let $B_1 = H \cap A$ and $B_2 = (H+x) \cap A$. Then we have that \begin{equation}\label{start}|A+q \cdot A| \geq |B_1 + q \cdot B_1| + |B_1 + q \cdot B_2| + |B_2 + q \cdot B_1| + |B_2 + q \cdot B_2|.\end{equation} Suppose, without loss of generality, that $|B_1| \geq |B_2|$. We separately consider several cases.

$(i)$ Suppose $B_1$ has rank 2. Then by Theorem \ref{main}, we have $|B_1 + q \cdot B_1| \geq (|q|+3)|B_1| -O(1)$. Furthermore by Theorem \ref{Ru}, we have $|B_1 + q \cdot B_2| + |B_2 + q \cdot B_1| \geq 2(|B_1| + 2 |B_2| - 3)$. Lastly, by \eqref{1}, we have $|B_2 + q \cdot B_2| \geq (|q|+1)|B_2| - O(1)$. Combining this three inequalities with \eqref{start} yields $|A+q \cdot A| \geq (|q|+5)|A| - O(1)$. Note that this case applies when $B_2$ consists of a single point.

$(ii)$ Suppose $B_1$ has rank 1 and $B_2$ has rank 2. By \eqref{1}, $|B_1 + q \cdot B_1| \geq (|q|+1)|B_1| - O(1)$ and by Theorem \ref{main}, $|B_2 + q \cdot B_2| \geq (|q|+3)|B_2| - O(1)$. We have that $B_1$ lies in a translate of some line, say $\ell$. Suppose $B_2$ lies in some distinct lines $x_1 + \ell , \ldots , x_m + \ell$ such that each $x_j + \ell$ intersects $B_2$ in at least one point. Note that $m \geq 2$ since $A$ has rank 3. For each $1 \leq j \leq m$, let $B_2^j = B_2 \cap (x_j + \ell)$. Then by the one dimensional Theorem \ref{Ru}, we have $$|B_1 + q \cdot B_2| \geq \sum_{j=1}^m |B_1 + q \cdot B_2^j| \geq m|B_1| + \sum_{j=1}^m (|B_2^j| - 1) \geq 2 |B_1| + |B_2| - 2.$$ Similarly, $|B_2 + q \cdot B_1| \geq 2|B_1| + |B_2| -2$. Combining these four inequalities with \eqref{start}, we obtain $|A+ q \cdot A| \geq (|q|+5)|A| - O(1)$. 

$(iii)$ Suppose $B_1$ and $B_2$ are both rank 1. Then the sets $x + q \cdot B_1$ and $B_1 + q \cdot x$ where $x \in B_2$ are all disjoint. Using \eqref{1}, we obtain (the extremal case being $|B_2|=2$)
$$|A+ q \cdot A| \geq (|q|+1)|A| - O(1) + 2|B_1||B_2| \geq (|q|+5)|A| - O(1).$$
\end{proof}

We now have to consider the case where $A$ is contained in four parallel lines. 

\begin{lem}
Suppose $A$ is contained in four parallel lines. Then $$|A+q \cdot A| \geq (|q|+5)|A| - O(1).$$
\end{lem}
\begin{proof}
Suppose $A$ is contained in four parallel lines, all parallel to some line through the origin $\ell$. Then $\mathbb{Z}^3 / \ell \cong \mathbb{Z}^2$ and say $A' = \{x_1 , x_2 , x_3 , x_4\} \subset \mathbb{Z}^3 / \ell$ are the 4 cosets that intersect $A$. Note that $A'$ must be a 2 dimensional set since $A$ is 3 dimensional. We want to show $|A' + q \cdot A'| \geq 14$. By the argument of Lemma \ref{FD}, we may assume that $A'$ intersects at least 3 residue classes modulo $q \cdot (\mathbb{Z}^3 / \ell)$. If $A'$ intersects four residue classes, then $|A' + q \cdot A'| = 16$. Otherwise let $A'_1$ be the intersection of $A'$ with the residue class that contains 2 elements of $A'$. Since $A'$ is 2 dimensional, it is not an arithmetic progression, so $|A'_1 + q \cdot A'| \geq |A'_1| + |A'| = 6$. Then $|A' + q \cdot A'| = 8 + |A'_1 + q \cdot A'| \geq 14$.

Let $A = B_1 \cup \dots \cup B_4$ where $B_i = (\ell + x_i) \cap A$. Then $B_i + q \cdot B_j$ are all disjoint, if we drop at most two pairs $\{i,j\}$. We do not need to drop a pair in the form $\{i,i\}$ because an equation in the form $x_i+qx_i=x_j+qx_j$ is not possible in $A'$. That means, any set $B_i$ can appear in a dropped pair at most twice.  Then 
$$
|A+ q \cdot A| \geq \sum_{i=1}^4 |B_i + q \cdot B_i| + \sum_{i\neq j \text{not dropped}} |B_i + q\cdot B_j| \geq $$
$$\geq \sum_{i=1}^4 ((|q|+1)|B_i| - O(1)) + \sum_{i\neq j} (|B_i|+|B_j|-1) - 2|A| = (|q|+5)|A|-O(1).
$$



\end{proof}

Finally we can express the analog of Lemma \ref{f}. Note that the proof uses Theorem \ref{main} and Theorem \ref{Ba} rather than any induction, otherwise identical to the proof of Lemma \ref{f}.

\begin{lem}\label{f3}
Let $A\subset \mathbb Z^3$ of rank $3$, $B\subset A$ and suppose that the rank of $B$ is $1\leq f < 3$. Then $$|B + q \cdot A| \geq (|q|+5)|B| - O(1),$$ or $A$ is contained in two parallel hyperplanes or four parallel lines. 
\end{lem}


\begin{thebibliography}{9}

\bibitem{Ba} Balog A. and Shakan G., On the sum of dilations of a set, {\em Acta Arithmetica}, to appear. 

\bibitem{Bu} Bukh B. (2008), Sums of Dilates, {\em Combinatorics, Probability and Computing}, {\bf 17}, 627--639.

\bibitem{Ci} Cilleruelo J., Hamidoune Y. and O. Serra (2009), On sums of dilates, {\em Combinatorics, Probability and Computing}, {\bf 18}, 871--880.

\bibitem{Du} Du S.\,S, Cao  H.\,Q. and Sun Z.\,W. (2014), On a sumset problem for integers, {\em Electronic Journal of Combinatorics}, {\bf 21}, 1--25.

\bibitem{Fr} Freiman G., {\em Foundations of structural theory of set addition}, American Math Society, 1973. 



\bibitem{Ha} Hamidoune Y. and Ru\'{e} J. (2011), A lower bound for the size of a Minkowski sum of dilates, {\em Combinatorics, Probability and Computing}, {\bf 20}, 249--256.


\bibitem{Lj} Ljujic Z. (2013), A lower bound for the size of a sum of dilates, {\em Journal of Combinatorial Number Theory}, {\bf 5}, 31--51.



\bibitem{Ru}  Ruzsa I. (1994), Sum of sets in several dimensions, {\em Combinatorica}, {\bf 14}, 48--490.


\bibitem{St} Stanchescu Y. (2001), An upper bound for $d$-dimensional difference sets, {\em Combinatorica}, {\bf 21}, 591-595. 

\end{thebibliography}
\end{document}